\RequirePackage{fix-cm}
\documentclass[smallextended]{svjour3}       
\usepackage{amsfonts, amsmath}

\usepackage{enumerate}
\usepackage{hyperref}
\usepackage{pictex, dcpic}
\usepackage{color}
\usepackage{graphicx}
\DeclareMathOperator{\im}{im}

\begin{document}
\def\soC{{\mathfrak{so}}}
\def\su{{\mathfrak{su}}}
\def\gotg{{\mathfrak{g}}}
\def\goth{{\mathfrak{h}}}
\def\gotm{{\mathfrak{m}}}
\def\BIp{\beta}
\def\Hp{\gamma}
\def\Frac[#1/#2]{\frac{#1}{#2}}
\def\R{\mathbb{R}}

\def\Cl{\hbox{Cl}}
\def\End{\hbox{End}}
\def\Spin{\hbox{Spin}}
\def\spin{\mathfrak{spin}}
\def\SO{\hbox{SO}}
\def\GL{\hbox{GL}}
\def\SU{\hbox{SU}}
\def\Ad{\hbox{Ad}}
\def\diag{\hbox{diag}}
 \def\one{\mathbb{I}}

\title{Barbero--Immirzi connections and how to build them
}


\author{A.\ Orizzonte \and L.\ Fatibene 
}


\institute{A.\ Orizzonte \at
             Dipartimento di Matematica, University of Torino, via Carlo Alberto 10, 10123 Torino (Italy) \\
              \email{andrea.orizzonte@unito.it}           
          \and
            L.\ Fatibene \at
             Dipartimento di Matematica, University of Torino, via Carlo Alberto 10, 10123 Torino (Italy) \\
             Istituto Nazionale di Fisica Nucleare (INFN), Sezione di Torino, via P. Giuria 1, 10125 Torino, (Italy)
             \email{lorenzo.fatibene@unito.it}
}

\date{\today}

\maketitle

\begin{abstract}
We introduce a covariant formulation of Barbero--Immirzi connections, which are used in Loop Quantum Gravity to describe gravity. We show that Barbero--Immirzi connections can be uniquely defined out of a given spin connection for any $(n+1)$-dimensional lorentzian manifold which is spin. A remarkable result is that the presence of a real Barbero--Immirzi parameter is a feature unique to the $4$-dimensional case.

\end{abstract}

\section{Introduction and Results}

The {Barbero--Immirzi connection} is used in Loop Quantum Gravity (LQG) to deal with gravity in lorentzian signature $(3,1)$. 

We briefly recall the original definition given by Barbero \cite{Barbero} and Immirzi \cite{Immirzi}. Fix a $4$-dimensional orientable lorentzian manifold $M$ with an Einstein metric $g$ and an embedded $3$-dimensional submanifold $S$ that is spacelike, so that the pull-back metric $h = t^* g$ along the embedding map $t \colon S \hookrightarrow M$ is positive definite.
Choose local coordinates $\left\{ s^A \right\}_{A = 1,2,3}$ on $S$ and a basis $\left\{ \rho_i \right\}_{i=1,2,3}$ for $\su(2)$. An $\su(2)$-valued {\it Barbero--Immirzi (BI) connection} is a $1$-form on $S$ with local coefficients
\begin{equation}
A^k_A(\BIp) = \Frac[1/2] \epsilon_{ij}{}^k \, \Gamma^{ij}_A + \BIp K^k_A, \quad \BIp \in \R
\label{biCoeff}
\end{equation}
where $\epsilon_{ij}{}^k$ is the totally antisymmetric Levi--Civita symbol, $\Gamma^{ij}_A$ are the Christoffel symbols of $h$ and $K^k_A$ are related to the coefficients of the Weingarten operator of $S$ (see \cite{KobaNu}, vol.\ 2, p.\ 14). The real number $\BIp$ is the {\it Barbero--Immirzi (BI) parameter} and for each choice of $\BIp$ one defines a different BI connection $A(\BIp) = A^k_A(\BIp) \ ds^A \otimes \rho_k$. In LQG the BI parameter has to be fixed by experimental data, so that one can speak of {\it the} BI connection $A$. 


There are several issues with the construction above. First, we want to use the BI connection to reformulate General Relativity, therefore we want to define $A$ without fixing an Einstein metric $g$ a priori. Second, the definition of the BI connection $A$ in terms of $\Gamma$ and $K$ looks arbitrary and there is no apparent explanation for the BI parameter $\BIp$. Third, the construction is carried out on a submanifold $S$ of $M$. As Samuel pointed out in \cite{Samuel} one would like to relate the BI connection $A$ to a connection defined on all of $M$.

The aim of this work is to give a geometrically well-defined method for building BI connections, thus solving all the issues outlined above. We will work in the general case of an $(n+1)$-dimensional lorentzian manifold $M$ and discuss the special case of interest $n = 3$. 

Ultimately, one is interested in reformulating Einstein's General Relativity as a variational theory for the BI connection $A$ instead of the metric $g$, which is a form more suitable for quantization and the starting point of LQG (see \cite{RovelliBook}). This is one of the main motivations behind the present work.

In Section 2 we review spin structures and spin frames, in this framework we can reformulate the problem without fixing a specific lorentzian metric $g$ on $M$.

In Section 3 we tackle the problem of reducing a principal $\Spin(n,1)$-bundle over $M$ to a $\Spin(n)$-bundle. The main idea is to define a $\Spin(n)$-connection $A$ out of a $\Spin(n,1)$-connection $\omega$. In general, reduction of a principal $G$-bundle over $M$ to a closed Lie subgroup $H \subset G$ is subject to topological obstructions. In the case of interest we prove that there are no obstructions, so that a reduction always exists.

In Section 4 we use the concept of reductive pair to explicitly define the BI connection $A$ on the $\Spin(n)$-bundle out of a connection $\omega$ on the $\Spin(n,1)$-bundle. We prove that $(\Spin(n,1), \Spin(n))$ is a reductive pair for $n \geq 3$ and that the existence of the BI parameter $\BIp$ is related to unique properties of the case $n = 3$. The remarkable result is that there is a {\it unique} way of defining the coefficients of $A$, which is precisely the one given in \eqref{biCoeff}.

\section{Spin Structures and Spin Frames}

Let $M$ be an $m$-dimensional real, smooth manifold and fix a signature $\eta=(r, s)$ on it, i.e.\ with $r+s=m$.
It is well known that all manifolds allow riemannian metrics, i.e.\ in signature $(m,0)$ (see \cite{KobaNu}, vol.\ 1, p.\ 116), whereas a pseudo-riemannian metric of signature $(r,s)$ exists if and only if the tangent bundle $TM$ splits as $TM=T_r\oplus T_s$ where $T_r, T_s$ are two subbundles of rank $r$ and $s$, respectively (see \cite{Steenrod}, p.\ 207).

It is also well known that whenever the manifold $M$ is orientable and admits a global metric of signature $(r,s)$, its frame bundle $L(M)$ admits a reduction to the group $\SO(r,s) \subset \GL(m)$.
Any such reduction is in fact associated to a metric $g$ of signature $(r,s)$, that is, any $\SO(r,s)$-reduction of the frame bundle $L(M)$ is given by a positive orthonormal frame bundle $SO(M, g)$ for some pseudo-riemannian metric $g$ on $M$ (see \cite{KobaNu}, vol.\ 1, p.\ 60).

A further topological obstruction is required for having spin structures on an orientable pseudo-riemannian manifold $(M, g)$. 
A {\it spin structure} (see, for instance, \cite{Lawson}, p.\ 80) is a pair $(P, \hat \ell)$, where $P$ is a principal bundle on $M$ with group $\Spin(r,s)$ (namely, the connected component of the identity in the spin group of the right signature) and $\hat\ell \colon P \to SO(M, g)$ is a principal bundle morphism with respect to the double cover $\ell \colon \Spin(r,s) \to \SO(r,s)$ which projects over the identity on $M$, i.e.\
\begin{equation}
\begindc{\commdiag}[10]
\obj(100,30)[M2]{$M$}
\obj(170,30)[M3]{$M$}
\obj(100,80)[P2]{$ P$}
\obj(170,80)[P3]{$SO(M,g)$}
%
\mor(100,29)(170,29){}[\atleft, \solidline] \mor(100,31)(170,31){}[\atleft, \solidline]
\mor{P2}{M2}{$p$}
\mor{P3}{M3}{$\pi$}
\mor{P2}{P3}{$\hat\ell$}
\enddc
\end{equation}
The obstruction to the existence of a spin structure over $(M, g)$ is the vanishing of the second Stiefel--Whitney class of $M$; see \cite{Antonsen}.

There is also a more general and somehow more geometric framework for spin structures, based on the universal covering of the frame bundle. It is interesting to briefly review it and also to compare to it; see \cite{Dabrowski}.

Whenever the manifold $M$ is orientable we can reduce the structure group of $L(M)$ to $\GL_+(m)$, the connected component of the identity in $\GL(m)$, to get the bundle of positively oriented frames $L_+(M)$. The double cover of $L_+(M)$ is denoted by $\widetilde L_+(M)$ and is a principal bundle with structure group $\widetilde{\GL}_+(m)$, the double cover of $\GL_+(m)$, the covering map is $\tilde \pi \colon \widetilde L_+(M) \to L_+(M)$. A spin structure $\Sigma$ is then defined to be a $\Spin(r,s)$-subbundle of $\tilde L_+(M)$.

Fixing a metric $g$ of signature $(r,s)$ on $M$ defines the subbundle $SO(M, g)$ of positive $g$-orthonormal frames. Consider the embedding $\iota \colon SO(M, g) \hookrightarrow L_+(M)$ and define the principal $\Spin(r,s)$-bundle $\Sigma_g := \iota^* SO(M,g)$. We have the commutative diagram
\begin{equation}
\begindc{\commdiag}[10]
\obj(100, 30)[SO]{$SO(M,g)$}
\obj(100, 80)[S]{$\Sigma_g$}
\obj(170, 30)[L]{$L_+(M)$}
\obj(170, 80)[LL]{$\widetilde L_+(M)$}
\mor{SO}{L}{$\iota$}[\atleft, \injectionarrow]
\mor{S}{SO}{}
\mor{S}{LL}{$\tilde \iota$}[\atleft, \injectionarrow]
\mor{LL}{L}{$\tilde \pi$}
\enddc
\end{equation}
so that $(\Sigma_g, \tilde \pi \circ \tilde\iota)$ is a standard spin structure.

In this more general framework, one does not need to fix a metric a priori. On the contrary we can consider all spin structures at once. In particular, one can lift a diffeomorphism of $M$ to $L(M)$ by using its natural bundle structure, and to $\widetilde L_+(M)$ by using the covering properties. As a result, in this framework one can systematically define isomorphisms of spin structures, possibly with different underlying metrics and define natural-like properties. However, all spin structures end up in the same framework, even those with non-isomorphic $\Spin(r,s)$-bundles. This generalization is not necessary whenever we have a fixed metric $g$, while it is needed when discussing properties of all spin structures at once (e.g.\ in a variational setting).

According to this standard setting, we give the definition of spin manifold.

\begin{definition}
A real, smooth, orientable manifold $M$ which allows global metrics of signature $(r,s)$ is called {\it spin manifold} if it admits a spin structure $(P, \hat \ell)$.
\end{definition}

\begin{remark}
From now on we shall fix a signature $(r,s)$ and consider a spin manifold $M$. If $M$ admits a metric $g$, we will also call $(M,g)$ a riemannian manifold, regardless of the signature $(r,s)$.
\end{remark}

Although this setting is very well suited to discuss properties on a fixed riemannian structure $(M, g)$, it is considerably less suited to discuss the gravitational theory in which the metric structure is unknown until one solves the Einstein equations (possibly coupled to other equations). In other words, one needs an equivalent formulation of spin structures, though in the category of spin manifolds rather than in the category of riemannian manifolds. 

\begin{definition}
Consider the double cover $\ell \colon \Spin(r,s) \rightarrow \SO(r,s)$ and the canonical embedding $\iota_{r,s} \colon \SO(r,s) \hookrightarrow \GL(m)$. Given a spin manifold $M$ in a signature $(r,s)$, a {\it spin frame} is principal bundle map $e \colon P \to L(M)$ with respect to the group morphism $\iota_{r,s} \circ \ell$, from some principal $\Spin(r,s)$-bundle $P$ on $M$ into the frame bundle $L(M)$. The commutative diagram is
\begin{equation}
\begindc{\commdiag}[10]
\obj(100, 30)[M1]{$M$}
\obj(100, 80)[P]{$P$}
\obj(170, 30)[M2]{$M$}
\obj(170, 80)[L]{$L(M)$}
\mor(100,29)(170,29){}[\atleft, \solidline] \mor(100,31)(170,31){}[\atleft, \solidline]
\mor{P}{M1}{$p$}
\mor{P}{L}{$e$}
\mor{L}{M2}{$\pi$}
\enddc
\end{equation}
\end{definition}

Denote by $\eta$ the canonical bilinear form on $\R^m$ which is represented in the canonical basis by the diagonal matrix of signature $(r,s)$, namely $\diag(\eta)= \diag(-1, \dots, -1, 1, \dots, 1)$ with $r$ pluses and $s$ minuses. Fix a trivialization on $P$ over an open set $U\subset M$, given by a local section $\sigma$. We have
\begin{equation}
e(\sigma(x)) = E_x \in L_x(M), \quad \forall x \in U
\end{equation}
Notice that, by equivariance, the values $E_x$ completely determine $e$ on $p^{-1}(U) \simeq U \times \Spin(r,s)$.

Since $E_x$ is an isomorphism $E_x \colon \R^m \rightarrow T_xM$ then, for any $X,Y \in T_xM$, the metric $g$ associated to $e$ at $x \in M$ is
\begin{equation}
g_x(X,Y) = \eta((E_x)^{-1}(X), (E_x)^{-1}(Y))
\end{equation}
Reasoning as in \cite{KobaNu}, vol.\ 1, p.\ 60, example 5.7, shows that $g$ is a well-defined smooth, symmetric, non-degenerate metric tensor of signature $(r,s)$ on all $M$.

Notice how a spin frame $e \colon P \rightarrow L(M)$ disentangles the metric structure $g$ from the spin structure: by considering only $P$ we can reduce the structure group without needing to fix a metric $g$. In this sense, the spin frame is a richer and more flexible structure than the metric $g$.

%
%
%

\section{Reduction of $\Spin(n,1)$-bundles to $\Spin(n)$}

From this section onwards we will fix $\dim M = m = n + 1$ and restrict to lorentzian signature, i.e.\ $(r,s) = (n,1)$. In the case of interest $n = 3$ we want to define an $\SU(2)$-connection out of a $\Spin(3,1)$-connection on $P$. This suggests that we should investigate the $\SU(2)$-reductions of a $\Spin(3,1)$-bundle $P$, we recall the definition:

\begin{definition}
Consider a principal $G$-bundle $p \colon P \rightarrow M$ and a closed Lie subgroup $H \subset G$. An {\it $H$-reduction of $P$} is a principal $H$-bundle $q \colon Q \rightarrow M$ with a principal bundle embedding
\begin{equation}
\begindc{\commdiag}[10]
\obj(30,30)[M1]{$M$}
\obj(100,30)[M2]{$M$}
\obj(30,80)[Q]{$Q$}
\obj(100,80)[P]{$P$}

\mor{M1}{M2}{}[\atleft, \solidline] \mor(30,33)(100,33){}[\atleft, \solidline]
\mor{Q}{M1}{$q$}
\mor{P}{M2}{$p$}
\mor{Q}{P}{}[\atleft, \injectionarrow]
\enddc
\end{equation}
When $p \colon P \rightarrow M$ admits an $H$-reduction we say that the structure group $G$ of the bundle can be reduced to $H$.
\end{definition}

Since $\SU(2) \simeq \Spin(3)$, the generalization to arbitrary dimension is to study $\Spin(n)$-reductions of a principal $\Spin(n,1)$-bundle $P$. We will also restrict to the case $n \geq 3$ since for $n = 1,2$ the spin groups are not the universal coverings of the respective special orthogonal groups.

The main result of this section is that for lorenztian spin manifolds the reduction from $\Spin(n,1)$ to $\Spin(n)$ is always possible.

\begin{theorem}\label{th1}
For any lorentzian spin manifold $M$ and any principal $\Spin(n,1)$-bundle $p \colon P \rightarrow M$, there always exists a $\Spin(n)$-reduction.
\end{theorem}

The proof relies on a definition and three lemmas.

\begin{definition} 
For a given principal $G$-bundle $p \colon P \rightarrow M$ and any closed Lie subgroup $H \subset G$, one has the (right) coset $G/H$ which is the set of elements $[g]$
\begin{equation}
[g] = \left\{ g' \in G : g' = gh \mbox{ for some } h \in H\right\}
\end{equation}

The group $G$ acts on the left on $G/H$ by
\begin{equation}
\lambda_H \colon G \times G/H \longrightarrow G/H \colon (g, [g']) \longmapsto [gg']
\end{equation}
The bundle associated to $P$ via this left action has $G/H$ as standard fiber, it is called {\it $H$-coset bundle of $P$} and is denoted by $p_H \colon P_H \rightarrow M$.
\end{definition}

\begin{lemma} (see, \cite{KobaNu}, vol.\ 1, p.\ 57)
The structure group of a principal $G$-bundle $p \colon P \rightarrow M$ is reducible to a closed Lie subgroup $H \subset G$ if and only if the associated $H$-coset bundle $p_H \colon P_H \rightarrow M$ admits a global cross section.
\end{lemma}

\begin{lemma} (see, \cite{KobaNu}, vol.\ 1, p.\ 58)
Consider a fiber bundle $\pi \colon B \rightarrow M$ with standard fiber $F$. If the fiber $F$ is diffeomorphic to $\R^k$, for some $k \in \mathbb{N}$, then the bundle admits a global cross section.
\end{lemma}

We want to prove \hyperref[th1]{Theorem 1} by showing that the coset $\Spin(n,1)/\Spin(n)$ is diffeomorphic to $\R^n$, we do so using the third lemma.

\begin{lemma} (see \cite{Steenrod}, p.\ 30)
Consider a smooth manifold $X$ and a transitive Lie group action $\lambda \colon G \times X \rightarrow X$, fix a point $x_0 \in X$. Let $H \subset G$ be the stabilizer subgroup of $x_0$, which is a closed Lie subgroup. If the map
\begin{equation}
\lambda_{x_0} \colon G \rightarrow X \colon g \mapsto g \cdot x_0
\end{equation}
is an open map, then
\begin{equation}
G/H \simeq X
\end{equation}
is a diffeomorphism.
\end{lemma} 

\begin{theorem}
The right coset ${}{\Spin(n,1)}/{\Spin(n)}$ is diffeomorphic to the euclidean space $\R^n$.
\end{theorem}
\begin{proof}
We prove this fact using the previous lemma. Consider the submanifold of $\R^{n+1}$
\begin{equation}
X = \left\{ (t, {\bf x}) \in \R^{n+1} : -t^2 + |{\bf x}|^2 = -1 , t > 0\right\}
\end{equation}
with the subspace topology. The action of $\Spin(n,1)$ is given by
\begin{equation}
S \cdot x = \ell(S) x, \quad \forall S \in \Spin(n,1)
\end{equation}
where $\ell \colon \Spin(n,1) \rightarrow \SO(n,1)$ is the two-to-one covering map and $\SO(n,1)$ acts through its fundamental representation.

By definition, an element $\bar S \in \Spin(n) \subset \Spin(n,1)$ acts only on the ${\bf x}$ part of a vector $v = (t, {\bf x})$. Therefore $\Spin(n)$ fixes the point $x_0 = (1, {\bf 0})$. Since $\Spin(n,1)$ is a path-connected Lie group and since
\begin{equation}
X = \Spin(n,1) \cdot x_0
\end{equation}
any open neighborhood $U$ of the identity $\one \in \Spin(n,1)$ is mapped into an open neighborhood of $x_0 \in X$. The hypotheses of the lemma above are all verified and we have
\begin{equation}
\Spin(n,1)/{\Spin(n)} \simeq X \simeq \R^n
\end{equation}
\end{proof}

\section{Barbero--Immirzi Connections through Reductive Pairs}

The main result of the previous section implies that for any $\Spin(n,1)$-bundle {$p \colon P \rightarrow M$} over a lorentzian spin manifold $M$ there exists a reduction to a $\Spin(n)$-bundle $q \colon Q \rightarrow M$. The commutative diagram is
\begin{equation}
\begindc{\commdiag}[10]
\obj(30,30)[M1]{$M$}
\obj(100,30)[M2]{$M$}
\obj(30,80)[Q]{$Q$}
\obj(100,80)[P]{$P$}

\mor{M1}{M2}{}[\atleft, \solidline] \mor(30,33)(100,33){}[\atleft, \solidline]
\mor{Q}{M1}{$q$}
\mor{P}{M2}{$p$}
\mor{Q}{P}{}[\atleft, \injectionarrow]
\enddc
\end{equation}


In this section we define the geometric framework for building BI connections on all $M$: given a principal connection $\omega$ on a $\Spin(n,1)$-bundle $p \colon P \rightarrow M$ and a $\Spin(n)$-reduction $q \colon Q \rightarrow M$, we define a unique principal connection $A$ on $Q$ out of $\omega$.

The main result relies on the definition of reductive pair of Lie groups.

\begin{definition}
(see \cite{KobaNu} vol.\ 1, p.\ 83) A pair of Lie groups $(G, H)$, with Lie algebras $\gotg$ and $\goth \subset \gotg$ respectively, is a {\it reductive pair} if: 

\begin{enumerate}[(i)]
\item $H \subset G$ is a closed Lie subgroup;

\item there exists a {\it reductive splitting}, i.e.~a vector space splitting
\begin{equation}
{\gotg} = {\goth} \oplus {\gotm} \quad \hbox{and} \quad \Ad_G(H) \, {\gotm} \subseteq {\gotm}
\end{equation}
\end{enumerate}
where $\Ad_G(H)$ is the adjoint representation of $G$ restricted to $H$.
\end{definition}

\begin{remark}
Notice that the vector subspace $\gotm \subset \gotg$ in the previous definition need not be a Lie subalgebra.
\end{remark}

The definition of reductive pair is particularly useful due to the following theorem.

\begin{theorem}
(see, \cite{KobaNu}, vol.\ 1, p.\ 83)
Let $(G, H)$ be a reductive pair with Lie algebras $\gotg, \goth$ respectively. Consider a principal $G$-bundle $p \colon P \rightarrow M$ with an $H$-reduction $q \colon Q \rightarrow M$. Choose a connection $1$-form $\omega$ on $P$, call $A$ its $\goth$-component and $K$ its $\gotm$-component. 

With this splitting the restriction the $A|_Q$ is a connection $1$-form on $Q$ while the restriction $K|_Q$ is a tensorial $1$-form of type $(\Ad_G, \gotm)$.
\end{theorem}

\begin{remark}
For a fixed reductive splitting $\gotg = \goth \oplus \gotm$ as in the previous theorem, we have a one-to-one correspondence between connections $\omega$ on $P$ and pairs $(A,K)$ on $Q$.
\end{remark}

We can now state the fundamental result of this work.

\begin{theorem}\label{th4}
If $n >  3$,  $(\Spin(n,1), \Spin(n))$ has a unique reductive splitting
\begin{equation}
\spin(n,1) = \spin(n) \oplus \gotm_0,
\end{equation}

If $n =3 $, we have a $1$-parameter family of reductive splittings
\begin{equation}
\spin(3,1) = \spin(3) \oplus \gotm_\BIp
\label{BIReductiveSplitting}
\end{equation}
The parameter $\BIp$ is called the {\it Barbero--Immirzi parameter} of the splitting. 
\end{theorem}

In the proof the algebras $\spin(n,1)$ and $\spin(n)$ are seen as subalgebras of the Clifford algebra, let us recall the basic definitions and fix our notations.


Consider $V = \R^m$ equipped with a non-degenerate, symmetric bilinear form $\eta$ of signature $(r,s)$ and the relative Clifford algebra $\Cl(V,\eta)$. We have an injective map $V \hookrightarrow \Cl(V,\eta)$, for any vector $v \in V$ we denote with the boldface ${\bf v}$ its image in the Clifford algebra.

The metric $\eta$ induces a non-degenerate bilinear form $Q$ on $\Cl(V,\eta)$, if $\eta$ has signature $(r,s)$ we define the spin group $\Spin(r,s)$ as the group generated by products of elements $S \in \Cl(V,\eta)$ such that $Q(S) = 1$. One can the prove that the adjoint action 
\begin{equation}
{\bf v} \longmapsto \Ad(S) {\bf v} = S \, {\bf v} \, S^{-1}, \quad \forall v \in V
\end{equation}
is an $\eta$-orthogonal transformation for any $S \in \Spin(r,s)$. 

Denote by $\ell \colon \Spin(r,s) \rightarrow \SO(r,s)$ the double covering of the spin group into the (connected component of the identity of the) special orthogonal group. The action of $\SO(r,s)$ on $V$ is assumed to be the standard representation. Since $V$ is mapped injectively into $\Cl(V,\eta)$ and since $\Ad(S) = \Ad(-S)$, we have that
\begin{equation}
\Ad(S) {\bf v} = \ell(S) v, \quad \forall v \in V
\end{equation}

Now choose an $\eta$-orthogonal basis for $V$ of vectors $\left\{ e_1, \dots, e_m \right\}$, the spin algebra $\spin(r,s)$ is the real vector subspace of $\Cl(V, \eta)$ generated by products of two elements, that is
\begin{equation}
\spin(r,s) = \langle {\bf e}_a {\bf e}_b : a,b = 1, \dots, m \mbox{ and } a \neq b \rangle 
\end{equation}
Recall that ${\bf e}_a {\bf e}_b = -{\bf e}_b {\bf e}_a$ so that only elements with $a < b$ form a basis. The adjoint representation of $\Spin(r,s)$ on its algebra $\spin(r,s)$ has the simple form
\begin{equation}
\Ad(S) \colon {\bf e}_a {\bf e}_b \longmapsto S \, {\bf e}_a {\bf e}_b \, S^{-1}, \quad \forall S \in \Spin(r,s)
\end{equation}
where the products are in $\Cl(V,\eta)$.

Let us now return to the case of lorentzian signature $(n,1)$. In $V = \R^{n+1}$ it is customary to denote the $\eta$-orthogonal basis elements by $e_0, e_1, \dots, e_n$, that is
\begin{equation}
\begin{aligned}
&\eta(e_0, e_0) = -1
	\\
&\eta(e_i, e_j) = \delta_{ij}, \quad i,j = 1, \dots, n
	\\
&\eta(e_0, e_i) = 0, \quad i = 1,\dots, n
\end{aligned}
\end{equation}
Define $W = \R^n$ as the span of the last $n$-vectors $e_1, \dots, e_n$ and denote by $\delta$ the restricted metric, which is euclidean. We have an isometric embedding $\iota \colon (W,\delta) \rightarrow (V,\eta)$. From here onwards, latin indices from the beginning of the alphabet (e.g.\ $a,b,c$) will range from $0$ to $n$ while latin indices from the middle of the alphabeth (e.g.\ $i,j,k$) will range from $1$ to $n$.

By functoriality of Clifford algebras (see \cite{Lawson}) we have an injection $\Spin(n) \hookrightarrow \Spin(n,1)$ which is characterised by
\begin{equation}
\Ad(S) {\bf e}_0 = {\bf e}_0, \quad \forall S \in \Spin(n)
\end{equation}

Similarly, we have an injection $\spin(n) \hookrightarrow \spin(n,1)$ through which the generators ${\bf e}_i {\bf e}_j$ are mapped to themselves. Therefore the basis for $\spin(n)$ can be completed to a basis for $\spin(n,1)$ by adding the vectors $\left\{ {\bf e}_0 {\bf e}_k \right\}_{k = 1, \dots, n}$.

We can now prove \hyperref[th4]{Theorem 4}.

\begin{proof}
Let $U:={\spin(n,1)} /{\spin(n)}$. There is a short exact sequence
\begin{equation}
\begindc{\commdiag}[10]
\obj(-30,30)[o1]{$0$}
\obj(20,30)[h]{$\spin(n)$}
\obj(80,30)[g]{$\spin(n,1)$}
\obj(140,30)[m]{$U$}
\obj(180,30)[o2]{$0$}

\mor{o1}{h}{}
\mor{h}{g}{}
\mor{g}{m}{}
\mor{m}{o2}{}
\enddc
\end{equation}
Identifying $U$ with $W$ via the isomorphism $ [{\bf e} _0 {\bf e}_k] :={\bf e} _0 {\bf e}_k + \spin (n) \mapsto e_k$, a splitting   corresponds to a linear injection $\phi \colon W \hookrightarrow \spin( n,1 ) $ such that $\phi (W)  \oplus  \spin (n) =\spin(n+1) $.
The most general choice of $\phi$  is 
\begin{equation}
   \phi({e}_k) = {\bf e}_0 {\bf e}_k + \psi (e_k )
\end{equation} 
for some linear map $\psi \colon W \rightarrow \spin(n)$.  Setting $\gotm := \phi(W)$, we obtain the splitting 
\begin{equation}
\spin(n,1) = \spin(n) \oplus \gotm.
\end{equation}
We now determine under which conditions the vector space $\gotm$ is $\Ad(\Spin(n))$-invariant. The space $\gotm$ is spanned by vectors of the form ${\bf e} _0 {\bf e} _k  + \psi (e_k ) $. For any $S \in \Spin(n)$ the element
\begin{equation}
\begin{aligned}
\Ad(S)({\bf e} _0 {\bf e} _k  + \psi ({e} _k )) 
&=S {\bf e } _0 S^{-1} \, S {\bf e} _k S^{-1} + (\Ad (S) \circ \psi) ({e} _k )
	\\
&={\bf e } _0 \, \Ad (S) ({\bf e} _k ) + (\Ad (S) \circ \psi) ({e} _k )
\end{aligned}
\end{equation}
belongs to $\gotm $ if and only if 
\begin{equation}
(\Ad (S) \circ \psi )({e} _k ) = (\psi \circ \Ad (S) )({\bf e } _k ).
\end{equation}
Since $\Ad (S) ({\bf e } _k )  = \ell (S) ({ e } _k ) $,
the splitting is reductive if and only if
$\Ad(S) \circ \psi = \psi \circ \ell(S)$,
that is if $\psi$ is an intertwiner between the $\Ad(S) \in \End(\spin(n))$ and $\ell(S) \in \End(W)$ representations of $\Spin(n)$.

If $n \neq 4$ the group $\SO(n) $ acts irreducibly on $W$. The adjoint representation of $\spin(n) $ on itself is irreducible since $\soC(n) =\spin (n) $ is a simple Lie algebra, hence the adjoint representation of $\Spin (n) $ on $\spin (n) $ is also irreducible. Given that both $\Ad$ and $\ell$ are irreducible representations of $\Spin(n)$, by Schur's Lemma $\psi$ is either the null map or an isomorphism. In the latter case we must have $n = \dim W = \dim( \spin(n) )= n(n-1)/2$
which is possible only if  $n = 0$ or $n = 3$.

For $n =3 $ we  identify $W $ with $  \spin (3) \simeq \su(2)$ via the map  $ {e} _k  \mapsto \tau_k$, where $\left\{\tau_k\right\}_{k = 1,2,3}$ are the Pauli matrices multiplied by the immaginary unit $i$. Then $\psi$ becomes an equivariant map with respect to the adjoint representation of $\Spin (3) $ on $\spin (3) $. Since any endomorphism of an odd dimensional real vector space has a real eigenvalue, it follows by Schur's lemma that $\psi$ is a constant multiple of the identity, i.e.\
\begin{equation}
\psi({e}_k) = \BIp \, \tau_k, \quad \BIp \in \R
\end{equation}

For $n = 4$ we use the fact that $\Spin(4) \simeq \SU(2) \times \SU(2)$ and $\spin(4) \simeq \su(2) \oplus \su(2)$. Each of the two copies of $\su(2)$ is a $3$-dimensional invariant subspace for the adjoint action of $\Spin(4)$ on $\spin(4)$. On the other hand, as shown in \cite{Lawson}, the fundamental representation $\ell$ has two invariant subspaces which are both $2$-dimensional. 

Since $\ker \psi$ is an invariant subspace for $\ell$, its dimension must be $0, 2$ or $4$. By the rank-nullity theorem $\dim W = \dim (\ker \psi) + \dim (\im \psi)$, so that $\dim (\im \psi)$ is $4, 2$ or $0$. However $\im \psi$ is an invariant subspace for $\Ad$, so that its dimension is $0, 3$ or $6$. The only possibiity therefore is $\dim(\ker \psi) = 4$ which implies that $\psi = 0$. 

The proof is thus complete.
\end{proof}

\begin{corollary}\label{cor1}
For $n > 3$ we have the unique reductive splitting $\spin(n,1) = \spin(n) \oplus \gotm_0$ with
\begin{equation}
\gotm_0 = \langle {\bf e}_0 {\bf e}_k : k = 1, \dots, n \rangle
\end{equation}

For $n = 3$ we have a one-parameter family of reductive splittings $\spin(3,1) = \su(2) \oplus \gotm_\BIp$ with
\begin{equation}
\gotm_\BIp = \langle \sigma_k + \BIp \tau_k : k = 1,2,3\rangle, \quad \BIp \in \R
\end{equation}
where $\left\{\sigma_k\right\}_{k=1,2,3}$ are the Pauli matrices and $\tau_k = i \sigma_k$.
\end{corollary}

\

We close this section by recovering the expression for the coefficient of a BI connection $A$ in the case $n = 3$. Choose a basis $\left\{ e_a \right\}_{a = 0,\dots 3}$ in $\R^4$ and define
\begin{equation}
T_{ab} = \Frac[1/2] e_a \wedge e_b
\end{equation}
Since $\Lambda^2 \R^4 \simeq \soC(3,1) \simeq \spin(3,1)$ (see \cite{Lawson} for details), the bivectors $\left\{ T_{ab} \right\}_{a < b}$ form a basis for the spin algebra $\spin(3,1)$. For any vector $v \in \R^4$ we have that $T_{ab}$ acts as
\begin{equation}
T_{ab}(v) = \Frac[1/2] (\eta(e_a, v) e_b - \eta(e_b, v) e_a)
\end{equation}
By looking at all commutators $[T_{ab}, T_{cd}]$ one finds that
\begin{equation}
T_{0k} = -\Frac[1/4] \sigma_k \quad \mbox{and} \quad T_{ij} = \Frac[1/4] \epsilon_{ij}{}^k \tau_k
\end{equation}

Let us now return to the $\Spin(3,1)$-bundle $p \colon P \rightarrow M$ with a $\SU(2)$-reduction $q \colon Q \rightarrow M$. 
\begin{equation}
\begindc{\commdiag}[10]
\obj(30,30)[M1]{$M$}
\obj(100,30)[M2]{$M$}
\obj(30,80)[Q]{$Q$}
\obj(100,80)[P]{$P$}

\mor{M1}{M2}{}[\atleft, \solidline] \mor(30,33)(100,33){}[\atleft, \solidline]
\mor{Q}{M1}{$q$}
\mor{P}{M2}{$p$}
\mor{Q}{P}{}[\atleft, \injectionarrow]
\enddc
\end{equation}
Let us also fix a BI parameter $\BIp$ so that we have a reductive splitting $\spin(3,1) = \su(2) \oplus \gotm_\BIp$. A connection $1$-form $\omega$ on $P$ can be written as
\begin{equation}
\omega = \omega^{ab} \otimes T_{ab}, \quad \mbox{with } \omega^{ab} \in \Omega^1(P)
\end{equation}
Following \hyperref[cor1]{Corollary 1}, the vectors $\left\{ -\Frac[1/2](\sigma_k + \BIp \tau_k) \right\}_{k=1,2,3}$ form a basis for the space $\gotm_\BIp$. Therefore we can split $\omega$ into $\su(2)$ and $\gotm_\BIp$ components
\begin{equation}
\begin{aligned}
\omega &= 2 \omega^{0k} \otimes T_{0k} + \omega^{ij} \otimes T_{ij} = 
	\\
&= 2\omega^{0k} \otimes \left(- \Frac[1/4] \sigma_k \right) + \omega^{ij} \otimes \left( \Frac[1/4] \epsilon_{ij}{}^k \tau_k \right) = 
	\\
&= \omega^{0k} \otimes \left( -\Frac[1/2](\sigma_k + \BIp \tau_k) \right) + \left( \Frac[1/2] \epsilon_{ij}{}^k \omega^{ij} + \BIp \omega^{0k} \right) \otimes \left( \Frac[1/2] \tau_k \right)
\end{aligned}
\end{equation}

By choosing $\left\{ \Frac[1/2] \tau_k \right\}_{k = 1,2,3}$ as a basis for $\su(2)$, we get that the coefficients of $A$ and $K$ with BI parameter $\BIp$ are
\begin{equation}
\begin{cases}
A^k = \Frac[1/2] \, \epsilon_{ij}{}^k \, \omega^{ij} + \BIp \, \omega^{0k}
	\\
K^k = \omega^{0k}	
\end{cases}
\end{equation}
We thus find that our construction generalizes the definition of BI connections to a generic $(n+1)$-dimensional lorentzian manifold $M$, for $n \geq 3$. The BI parameter $\BIp$ comes out naturally in the case $n = 3$ whereas for $n > 3$ it is absent.

\section{Conclusions and Perspectives}

In this work we gave a well-defined geometric framework for building BI connections on any given $(n+1)$-dimensional lorentzian spin manifold $M$. The construction generalises what is found in the physics and mathematical physics literature and improves it by not requiring a fixed metric $g$ on $M$ and producing BI connections on the whole spacetime manifold $M$. The much discussed BI parameter $\BIp$ comes out naturally and we proved that it is a feature {\it unique} to the $4$-dimensional case.

As stated in the introduction, the main motivation behind this work was to reformulate General Relativity: instead of using a metric $g$ and a spin connection $\omega$ as field variables, we rewrite the lagrangian in terms of a spin frame $e$ and the pair $(A,K)$ for a fixed BI parameter $\BIp$. Since this formulation is the classical starting point of LQG, it is worthwhile to study the variational Euler--Lagrange equations for $(e,A,K)$ and their relation to the Einstein field equations. The equations dependence on the BI parameter $\BIp$ is also of interest, especially when we couple General Relativity with other fields (e.g.\ Klein--Gordon, Yang--Mills, Dirac).

Another application of this work is the investigation of the holonomy group of a BI connection $A$, which plays a central role in LQG. Samuel's argument in \cite{Samuel} can be solved by looking at the relation between the holonomy group of $A$ and that of the original spin connection $\omega$, both on $M$ and on a spacelike submanifold $S$. One should also refer to the classification of metric lorentzian holonomies by Berger \cite{HolClass} and Leistner \cite{Leistner} to better characterize the BI connections, this was suggested to one of the authors in private form.

\section*{Acknowledgements}

This article is based upon work from COST Action (CA15117 CANTATA), supported by COST (European Cooperation in Science and Technology).

We also acknowledge the contribution of INFN (Iniziativa Specifica QGSKY), the local research project {\it  Metodi Geometrici in Fisica Matematica e Applicazioni (2019)} of Dipartimento di Matematica of University of Torino (Italy). This paper is also supported by INdAM-GNFM.

We would like to deeply thank Giovanni Russo, for reading the initial draft and giving invaluable advice on how to improve the presentation, and Anna Fino, for useful comments and discussions.



%
%



\end{document}